\documentclass[12pt,reqno]{amsart}

\usepackage{amsmath,amsthm,amscd,amsfonts,amssymb,graphicx,color}

\usepackage[bookmarksnumbered,colorlinks,plainpages]{hyperref}
\hypersetup{colorlinks=true,linkcolor=red,anchorcolor=green,citecolor=cyan,urlcolor=red,filecolor=magenta,pdftoolbar=true}

\usepackage{mathrsfs}

\newtheorem{theorem}{Theorem}[section]
\newtheorem{lemma}[theorem]{Lemma}
\newtheorem{proposition}[theorem]{Proposition}
\newtheorem{corollary}[theorem]{Corollary}
\theoremstyle{definition}
\newtheorem{definition}[theorem]{Definition}
\theoremstyle{remark}

\numberwithin{equation}{section}

\AtBeginDocument{{\noindent\small
This is a preprint of a paper whose final and definite form is with 
\emph{Banach J. Math. Anal.}, ISSN: 1735-8787. 
Submitted 21-Aug-2018; Article accepted on 30-Nov-2018.} 
\vspace{9mm}}

\begin{document}

\setcounter{page}{1}

\title[Functional characterizations of trace spaces in Lipschitz domains]{
Functional characterizations of trace spaces\\ in Lipschitz domains}

% -----

\author[S. Touhami, A. Chaira, \MakeLowercase{and} D. F. M. Torres]{
Soumia Touhami,$^1$ Abdellatif Chaira,$^1$ \MakeLowercase{and} Delfim F. M. Torres$^{2*}$}

\address{$^{1}$Universit\'{e} Moulay Ismail, Facult\'{e} des Sciences, 
Laboratoire de Math\'{e}matiques et leures Applications, 
Equipe EDP et Calcul Scientifique, BP 11201 Zitoune, 50070 Mekn\`{e}s, Morocco.}

\email{\textcolor[rgb]{0.00,0.00,0.84}{touhami16soumia@gmail.com; a.chaira@fs.umi.ac.ma}}

% -----

\address{$^{2}$Center for R\&{}D in Mathematics and Applications (CIDMA), 
Department of Mathematics, University of Aveiro, 3810-193 Aveiro, Portugal.}

\email{\textcolor[rgb]{0.00,0.00,0.84}{delfim@ua.pt}}

% -----

\dedicatory{This paper is dedicated to Professor Ronald G. Douglas}

% -----------------------------------

\let\thefootnote\relax\footnote{\ }



\date{Submitted: August 21, 2018; Accepted: November 30, 2018.
\newline \indent $^{*}$Corresponding author}

% -----------------------------------

\begin{abstract}
Using a factorization theorem of Douglas,  	
we prove functional characterizations of trace spaces 
$H^s(\partial \Omega)$ involving a family 
of positive self-adjoint operators. 
Our method is based on the use of a suitable operator 
by taking the trace on the boundary $\partial \Omega$ of a 
bounded Lipschitz domain $\Omega \subset \mathbb R^d$ 
and applying Moore--Penrose pseudo-inverse properties 
together with a special inner product on $H^1(\Omega)$. 
Moreover, generalized results of the Moore--Penrose 
pseudo-inverse are also established.
\end{abstract}

\maketitle

% -------------------------------------------

\section{\textbf{Introduction}}

The class of Lipschitz domains covers most cases that arise in applications 
of partial differential equations and the characterization of the trace 
spaces $H^s(\partial \Omega)$ on this class is an important tool 
in the analysis of boundary value problems (BVPs) \cite{MR3266243,MR2369871}.
The usual descriptions of the trace spaces, well investigated by Adams and Fournier \cite{Ad}, 
Dautray and Lions \cite{Dau}, Lions and Magenes \cite{LMa}, and McLean \cite{Mc}, 
use Fourier transforms and local diffeomorphisms of the domain into a half space. 
However, many other descriptions do exist \cite{Auch,MR3735591,MR3755723}. 
Let us emphasize that, in the case of Lipschitz domains, not all the ways of characterizing 
$H^s(\partial \Omega)$ make sense and, if they do, they are not necessarily equivalent.
Here we investigate BVPs in Lipschitz domains with boundary data in $H^s(\partial \Omega)$, 
which are well-recognized as difficult and challenging problems \cite{MR3840887,MR3797189}.

Our main results provide natural descriptions of the trace spaces 
$H^s(\partial \Omega)$ for $|s|\leq 1$ on a bounded Lipschitz domain 
$\Omega \subset \mathbb R^d$, $d\geq 2$, using the trace operator 
$\Gamma_s: H^{s+1/2}(\Omega) \longrightarrow H^{s}(\partial \Omega)$ 
and its Moore--Penrose pseudo-inverse. 
The obtained descriptions use a theorem due to Douglas \cite{Do},
which asserts the equivalence of (i) $range\ [A] \subset range\ [B]$, 
(ii) $A A^{*} \leq \lambda^2 B B^{*}$ for some $\lambda \geq 0$, and 
(iii) the existence of a bounded $C$ on $\mathcal{H}$ such that $A=BC$, where 
$A$ and $B$ are bounded operators on the Hilbert space $\mathcal{H}$
and $C$ is constructed in such a way that $null [A]= null [C]$. 
This factorization theorem of Douglas has shown to be an important
result in many different contexts and with many interesting
implications: see, e.g., \cite{MR3779663,MR3783547,MR3760326}
and references therein. For example, Douglas' theorem can be used to solve the 
strong Parrott theorem obtained by Foias and Tannenbaum
\cite{MR0972228} and it solves the $(2,1)$-Parrott problem 
with some restrictions, as well as the Bakonyi--Woerdeman theorem 
obtained by Bakonyi and Woerdeman in \cite{MR1145412}:
see \cite{MR3626684}. Here we construct a family of Hilbert spaces describing 
$H^s(\partial \Omega)$. The resulting spaces form an interpolation family
that make our characterizations rich enough to deal with 
the regularity of the domain when adopting boundary integral methods 
in the analysis of BVPs in Lipschitz domains.  
This makes the obtained descriptions particularly interesting.
For example, our results are useful to give a boundary formula 
for the solution of the Dirichlet problem for the Laplacian 
in Lipschitz domains with boundary data in $H^s$ for some values of $s$. 

The paper is organized as follows.
In Section~\ref{sec:02}, we provide 
necessary definitions, we fix main notation,
and we recall the important factorization theorem 
of Douglas (Theorem~\ref{thm:DT}), which plays an important
role in the proof of our main results.
In Section~\ref{sec:03}, we begin by defining
the Moore--Penrose inverse operator,
we recall one of its main properties
(Lemma~\ref{lemma:prel:01}),
as well as some Labrousse identities 
(Propositions~\ref{prop:3.1} and \ref{prop:3.2}). 
Then, we prove our first original results:
Proposition~\ref{OR:01} and Theorems~\ref{OR:02} and \ref{OR:03}.
Section~\ref{sec:04} is dedicated to recall the notion
of Sobolev spaces on Lipschitz domains.
In particular, we define Lipschitz continuous boundary 
and Lipschitz domain (Definition~\ref{def:LCB}),
and recall some fundamental properties of
the Sobolev spaces (Lemmas~\ref{lem:4.2} and \ref{lem:4.3}).
In Section~\ref{sec:05}, we recall some properties
of the trace (Lemma~\ref{lem:5.1}) and embedding operators
(Lemma~\ref{lem:5.2}), we obtain a Green's formula
(Corollary~\ref{prop:GF}), while new useful results
on trace (Theorems~\ref{prop:5.3}, \ref{thm:5.5} and \ref{thm:5.6})
and embedding operators (Theorem~\ref{thm:5.7}) are proved.
We end with Section~\ref{sec:06}, where our main results
providing functional characterizations of trace spaces 
in Lipschitz domains are given (Theorem~\ref{thm:6.2},
Proposition~\ref{prop:usesDT} and Theorems~\ref{thm:6.6} 
and \ref{prop2:usesDT}), as well as some new results
on equivalence of norms (Theorem~\ref{cor:usesDT}
and its Corollaries~\ref{cor:6.9} and \ref{cor:6.10}).

% -------------------------------------------

\section{\textbf{Definitions, notation and the factorization theorem of Douglas}}
\label{sec:02}

Let $\mathcal H_1$ and $\mathcal H_2$ be two Hilbert spaces with inner products 
$(\cdot,\cdot)_{\mathcal H_1}$ and $(\cdot,\cdot)_{\mathcal H_2}$ and associated 
norms $\|\cdot\| _{\mathcal H_1}$ and $\|\cdot\| _{\mathcal H_2}$, respectively.  
Let us first fix some notations. By $\mathcal L(\mathcal H_1, \mathcal H_2)$, 
we denote the space of all linear operators from $\mathcal H_1$ 
into $\mathcal H_2$ and $\mathcal L(\mathcal H_1, \mathcal H_1)$ 
is briefly denoted by $\mathcal L(\mathcal H_1)$. For an operator 
$A\in \mathcal L(\mathcal H_1,\mathcal H_2)$, 
$\mathcal{D}(A)$, $\mathcal R(A)$ and $\mathcal N(A)$ 
denote its domain, its range and its null space, respectively. 
For $A, B \in \mathcal L(\mathcal H_1, \mathcal H_2)$,  
$B$ is called an extension of $A$ if $ \mathcal  D(A)\subset  \mathcal D(B)$ 
and $Ax  =  Bx$ for all $x\in \mathcal D(A)$, and this fact is denoted by  
$A\subset B$. The set of all bounded operators from $\mathcal H_1$ 
into $\mathcal H_2$ is denoted by $\mathcal B(\mathcal H_1, \mathcal H_2)$, 
while $\mathcal B(\mathcal H_1, \mathcal H_1)$ is briefly denoted 
by $\mathcal B(\mathcal H_1)$. The set of all closed densely defined 
operators from $\mathcal H_1$ into $\mathcal H_2$ is denoted 
by $\mathcal C(\mathcal H_1, \mathcal H_2)$, and 
$\mathcal C(\mathcal H_1, \mathcal H_1)$ is denoted by
$\mathcal C(\mathcal H_1)$. For  $A \in \mathcal C(\mathcal H_1,\mathcal H_2)$, 
its adjoint operator is denoted by $A^*\in \mathcal C(\mathcal H_2,\mathcal H_1)$.  
A self-adjoint operator $A$ on a Hilbert space $\mathcal H$ 
is said to be positive (strictly positive) if $(Ax,x)_{\mathcal H} \geq 0$ for all  
$x \in \mathcal D(A)$ ($(Ax,x)_{\mathcal H} > 0$); in such case we write $A\geq 0$ ($A> 0$).
The next theorem, due to Douglas \cite{Do}, is our central tool.

\begin{theorem}[Douglas' theorem \cite{Do}] 
\label{thm:DT}
Let $\mathcal H$ be a Hilbert space and $A,B \in \mathcal B(\mathcal H)$ 
be two bounded operators. The following statements are equivalent:
\begin{enumerate}
\item $\mathcal R(A)\subset \mathcal R(B)$;
\item $AA^* \leq \mu \ BB^*$ for some $\mu \geq 0$;
\item there exists a bounded operator $C\in \mathcal B(\mathcal H)$ such that $A=BC$.
\end{enumerate}
Moreover, if previous items 1, 2 and 3 hold, then there exists 
a unique operator $C\in \mathcal B(\mathcal H)$ such that
\begin{enumerate}
\item[(a)] $\|C\|^2= \inf \{ \mu \ | \ AA^* \leq \mu \  BB^* \}$;
\item[(b)] $\mathcal N(C) =\mathcal N(A)$;
\item[(c)] $\mathcal R(C) \subset \overline{\mathcal R(B^*)}$.
\end{enumerate}
\end{theorem}

We make use of Theorem~\ref{thm:DT} of Douglas 
in the proof of our Lemma~\ref{lem:useDT},  
Proposition~\ref{prop:usesDT} 
and Theorems~\ref{prop2:usesDT} and \ref{cor:usesDT}.
	
% -------------------------------------------

\section{\textbf{New results about the Moore--Penrose inverse}}
\label{sec:03}

Let ${\mathcal H}_1$ and ${\mathcal H}_2$ be two Hilbert spaces, 
$A \in {\mathcal C}({\mathcal H}_1, {\mathcal H}_2)$ be a closed densely 
defined operator and $A^*$  its adjoint. The Moore--Penrose 
inverse of $A$, denoted by $A^\dagger$, is defined as the unique 
linear operator in  ${\mathcal C }({\mathcal H}_2, {\mathcal H}_1)$  
such that
$$
{\mathcal D}(A^\dagger)={\mathcal R}(A)
\oplus {\mathcal N}(A^*), 
\quad  \mathcal N(A^{\dagger })= \mathcal N(A^*) 
$$ 
and 
\begin{equation*}
\begin{cases}
AA^\dagger A=A, &\text{  }  \\
A^\dagger AA^\dagger=A^\dagger, &\text{}
\end{cases}
\qquad
\begin{cases}
AA^\dagger \subset P_{\overline{{\mathcal R}(A)}}, & \text{  }  \\
A^\dagger A \subset P_{\overline{{\mathcal R}( A^\dagger)}}, &\text{}
\end{cases}
\end{equation*}
where $P_{{\mathcal E}}$ denotes the orthogonal projection 
on the closed subspace ${\mathcal E}$. 
       
\begin{lemma}[See, e.g., \cite{MR0331659}]
\label{lemma:prel:01}
Let $A\in \mathcal B(\mathcal H_1,\mathcal H_2)$ be a bounded operator  
with closed range. Then, $A$ has a bounded Moore--Penrose inverse  
$A^{\dagger}\in \mathcal B(\mathcal H_2,\mathcal H_1)$.
\end{lemma}        

According to a fundamental result of von Neumann (see \cite{Gro}), 
for $A\in \mathcal C(\mathcal H_1,\mathcal H_2)$ the operators 
$(I+AA^*)^{-1}$ and $A^*(I+AA^*)^{-1}$ are everywhere defined and bounded. 
Moreover, $(I+AA^*)^{-1}$ is self-adjoint.
Also, the operators $(I+A^*A)^{-1}$ and $A(I+A^*A)^{-1} $ are everywhere 
defined and bounded, and $(I+A^*A)^{-1}$ is self-adjoint. 
Moreover, 
\begin{equation*}
(I+AA^*)^{-1} A \subset A(I+A^*A)^{-1}
\end{equation*}
and
\begin{equation*}
(I+A^*A)^{-1} A^*  \subset A^*(I+AA^*)^{-1}
\end{equation*}
(see \cite{Gro,L}). In the following, 
we state some useful identities due to Labrousse \cite{L}
and Labrousse and Mbekhta \cite{LM}. 

\begin{proposition}[See Lemma 2.5 and Corollary 2.6 of \cite{L}]  
\label{prop:3.1}
Let $A \in \mathcal C(\mathcal H_1, \mathcal H_2)$ 
and $B \in \mathcal C(\mathcal H_2,\mathcal H_1)$ be
such that $ B=A^{\dagger}$. Then,
\begin{enumerate}
\item $A(I+A^*A)^{-1} = B^*(I+BB^*)^{-1}$;
\item $(I+A^*A)^{-1}+(I+BB^*)^{-1}= I+P_{\mathcal N(B^*)}$;     
\item $A^*(I+AA^*)^{-1} = B(I+B^*B)^{-1}$;
\item $(I+AA^*)^{-1}+(I+B^*B)^{-1}= I+P_{\mathcal N(A^*)}$;
\item $(I+AA^*)^{-1}+(I+B^*B)^{-1}= I$  (if $A^*$ is injective);
\item $\mathcal N(A^*(I+AA^*)^{-1/2}) =\mathcal N(A^*) = \mathcal N(B)$.
\end{enumerate}
\end{proposition}

\begin{proposition}[See Proposition~1.7 of \cite{LM}] 
\label{prop:3.2}
Let  ${\mathcal H}_1, {\mathcal H}_2$ be two Hilbert spaces, 
$A \in {\mathcal B}({\mathcal H}_1, {\mathcal H}_2)$ 
and $B$ its Moore--Penrose inverse. Then,
one has 
\begin{equation*}
\|x\|_{\mathcal H_1}^2= \|B^*(I+BB^*)^{-1/2} 
x\|_{\mathcal H_2}^2 + \|(I+BB^*)^{-1/2} x\|_{\mathcal H_1}^2
\end{equation*}
for all $x\in \mathcal H_1$.
\end{proposition}

For extensive results and applications concerning the Moore--Penrose 
inverse concept, we refer the reader to \cite{MR3824748,Gro,L,LM}
and references therein. The following results are, 
to the best of our knowledge, new.

\begin{proposition} 
\label{OR:01}
Let ${\mathcal H}_1$ and ${\mathcal H}_2$ be two Hilbert spaces, 
$A \in {\mathcal B}({\mathcal H}_1, {\mathcal H}_2)$ 
and $B$ be its Moore--Penrose inverse. If $x\in \mathcal R(B)$, then
\begin{equation*}
\|x\|_{\mathcal H_1}^2 = \|(I+BB^*)^{-1/2} x\|_{\mathcal H_1}^2 
+ \|(I+A^*A)^{-1/2} x\|_{\mathcal H_1}^2.
\end{equation*}
\end{proposition}

\begin{proof}
For $x\in \mathcal R(B)=\mathcal N(B^*)^{\perp}$, where $\mathcal N(B^*)^{\perp}$ 
denotes the orthogonal complement of $\mathcal N(B^*)$, 
$$ 
(I+A^*A)^{-1}x+ (I+BB^*)^{-1}x = x
$$
according to the second item of Proposition~\ref{prop:3.1},
which implies that
\begin{equation*}
\begin{split}
\|x\|_{\mathcal H_1}^2 
&= (x,x)_{\mathcal H_1}
=\left((I+A^*A)^{-1}x+(I+BB^*)^{-1}x,x\right)_{\mathcal H_1}\\
&= \|(I+A^*A)^{-1/2}x\|_{\mathcal H_1}^2 
+ \|(I+BB^*)^{-1/2}x \|_{\mathcal H_1}^2.
\end{split}
\end{equation*}   
The proof is complete.
\end{proof} 

We will extensively make use of the following result:

\begin{theorem}   
\label{OR:02}
Let  ${\mathcal H}_1$ and ${\mathcal H}_2$ be two Hilbert spaces, 
$A \in {\mathcal B}({\mathcal H}_1, {\mathcal H}_2)$ and 
$B$ be its Moore--Penrose inverse. Then, the operator $B^*(I+BB^*)^{-1/2}$ 
is bounded with closed range and has a bounded Moore--Penrose inverse given by 
\begin{equation*}
T_B=B(I+B^*B)^{-1/2}+A^*(I+B^*B)^{-1/2}.
\end{equation*}
Moreover, the adjoint operator of $T_B$ is $T_{B^*}$, where 
\begin{equation*}
T_{B^*}=B^*(I+BB^*)^{-1/2}+A(I+BB^*)^{-1/2}.
\end{equation*}
\end{theorem}
    
\begin{proof}
For $x\in \mathcal H_1$ we have, 
according to Proposition~\ref{prop:3.2}, that
\begin{equation*}
\|x\|_{\mathcal H_1}^2= \|B^*(I+BB^*)^{-1/2} x\|_{\mathcal H_2}^2
+ \|(I+BB^*)^{-1/2} x\|_{\mathcal H_1}^2,
\end{equation*}
which implies that 
$$
\|B^*(I+BB^*)^{-1/2} x\|_{\mathcal H_2} \leq \|x\|_{\mathcal H_1}.
$$
Therefore, the operator $B^*(I+BB^*)^{-1/2}$ is bounded.  
To establish that it has a closed range, it suffices to prove 
the existence of a constant $c>0$ such that if $x\in \mathcal N 
\left(B^*(I+BB^*)^{-1/2}\right)^{\perp} = \mathcal N(B^*)^{\perp} 
= \mathcal R(B)$ (item 6 of Proposition~\ref{prop:3.1}), then 
$$ 
c \|x\|_{\mathcal H_1} \leq \|B^*(I+BB^*)^{-1/2} x \|_{\mathcal H_2}.
$$
In fact, for $x\in \mathcal R(B)$, we have according 
to our Proposition~\ref{OR:01} that
$$
\|x\|_{\mathcal H_1}^2 = \|(I+BB^*)^{-1/2} x\|_{\mathcal H_1}^2 
+ \|(I+A^*A)^{-1/2} x\|_{\mathcal H_1}^2.
$$
Combining the last equality with the one given 
by Proposition~\ref{prop:3.2}, we obtain that  
$$
\|B^*(I+BB^*)^{-1/2} x\|_{\mathcal H_2} 
= \|(I+A^*A)^{-1/2} x\|_{\mathcal H_1}
$$
for all $x\in \mathcal R(B)$. Because $I+A^*A$ 
and its inverse are positive and bounded, 
it follows that their associated square roots  
are positive and bounded as well. So, $(I+A^*A)^{-1/2}$ is positive, 
bounded, and has a bounded inverse. This implies the existence of a 
positive constant $c$ such that 
$$
c \ \|x\|_{\mathcal H_1} \leq \|(I+A^*A)^{-1/2}x \|_{\mathcal H_1}
$$ 
for all $x\in \mathcal H_1$. In particular, 
for $x\in \mathcal R(B)$, we have
$$
c \ \|x\|_{\mathcal H_1} \leq  \|(I+A^*A)^{-1/2}x \|_{\mathcal H_1}
=\|B^*(I+BB^*)^{-1/2}x \|_{\mathcal H_2}.
$$ 
Therefore, we deduce that $\mathcal R(B^*(I+BB^*)^{-1/2})$ is closed. 
Consequently, according to Lemma~\ref{lemma:prel:01}, 
$B^*(I+BB^*)^{-1/2}$ has a bounded Moore--Penrose inverse.
On the other hand, we verify that
\begin{multline*}
T_B B^*(I+BB^*)^{-1/2}
= \left(B(I+B^*B)^{-1/2}+A^*(I+B^*B)^{-1/2}\right) B^*(I+BB^*)^{-1/2}\\
=B(I+BB^*)^{-1/2}B^*(I+BB^*)^{-1/2}+ A^* (I+B^*B)^{-1/2}B^*(I+BB^*)^{-1/2}. 
\end{multline*}
Since  
\begin{gather*}
(I+B^*B)^{-1/2} B^*  \subset B^*(I+BB^*)^{-1/2},\\
B^*(I+BB^*)^{-1}=A(I+A^*A)^{-1}
\end{gather*}
and $A^* B^* \subset BA  = P_{\mathcal R( B)}$,
we obtain that
\begin{equation*}
\begin{split}
T_B B^*(I+BB^*)^{-1/2}  
&= BB^*(I+BB^*)^{-1}+ A^* B^*(I+BB^*)^{-1}\\
&= BA(I+A^*A)^{-1})+A^* B^*(I+BB^*)^{-1}\\
&= BA \big( (I+A^*A)^{-1} +(I+BB^*)^{-1} \big )\\
&= BA (I+P_{\mathcal N(A)})\\ 
&= BA = P_{\mathcal R( B)}.
\end{split}
\end{equation*}
Similarly, one can prove that
$$ 
B^*(I+BB^*)^{-1/2}T_B = P_{\mathcal R(B^*)}.
$$
This implies that
$$
T_{B} B^*(I+BB^*)^{-1/2} T_{B} = T_{B}
$$
and 
$$ 
B^*(I+BB^*)^{-1/2} T_{B}B^*(I+BB^*)^{-1/2} = B^*(I+BB^*)^{-1/2}.
$$
Therefore, $T_B$ is the Moore--Penrose inverse of $B^*(I+BB^*)^{-1/2}$.  
Moreover, since 
$$
(B(I+B^*B)^{-1/2})^* =B^*(I+BB^*)^{-1/2}
$$ 
and  
$$
(A^*(I+B^*B)^{-1/2})^*=A(I+BB^*)^{-1/2}
$$ 
(see \cite{L}), we obtain that
\begin{equation*}
\begin{split}
T_{B}^*
&=(B(I+B^*B)^{-1/2})^*+(A^*(I+B^*B)^{-1/2})^*\\
&= B^*(I+BB^*)^{-1/2}+A(I+BB^*)^{-1/2}\\
&=T_{B^*}.
\end{split}
\end{equation*}
The result is proved: $T_B^*=T_{B^*}$.
\end{proof}
    
From Propositions~\ref{prop:3.1}, \ref{prop:3.2} and \ref{OR:01}, 
one can easily deduce the two following corollaries.
    
\begin{corollary} 
\label{cor:3.6}
Let $A\in\mathcal B(\mathcal H_1,\mathcal H_2)$ 
and $B$ be its Moore--Penrose inverse. Then,
$B^*(I+BB^*)^{-1/2}$ is an isomorphism from 
$\mathcal N(B^*)^{\perp}$ to $\mathcal R(B^*)$. 
\end{corollary}
    
\begin{corollary}
Let $A\in\mathcal B(\mathcal H_1,\mathcal H_2)$ and $B$ be 
its Moore--Penrose inverse. Then, $T_B$ is an isomorphism from 
$\mathcal R(B^*)$ to $\mathcal N(B^*)^{\perp}$.  
\end{corollary}
        
Our next result provides a decomposition for an arbitrary  
bounded operator in terms of its Moore--Penrose inverse.

\begin{theorem}
\label{OR:03}
Let $\mathcal H_1$ and $\mathcal H_2$ be two Hilbert spaces, 
$A\in \mathcal B(\mathcal H_1,\mathcal H_2)$ and $B$ be 
its Moore--Penrose inverse. Then, the decomposition
$$
A= (I+B^*B)^{-1/2} T_{B^*}
$$
holds, where $T_{B^*}=B^*(I+BB^*)^{-1/2}+A(I+BB^*)^{-1/2}$.
\end{theorem}
      
\begin{proof}
We have
$$
(I+B^*B)^{-1/2}T_{B^*}
=(I+B^*B)^{-1/2} \left(B^*(I+BB^*)^{-1/2}+A(I+BB^*)^{-1/2}\right).
$$
Moreover, since 
$$
(I+B^*B)^{-1/2} B^* \subset B^* (I+BB^*)^{-1/2}
$$
and, from the third item of Proposition~\ref{prop:3.1}, 
$$ 
B^*(I+BB^*)^{-1}=A(I+A^*A)^{-1},
$$  
it follows that
\begin{equation*}
(I+B^*B)^{-1/2}T_{B^*} 
= B^*(I+BB^*)^{-1}+(I+B^*B)^{-1/2} A  (I+BB^*)^{-1/2}.
\end{equation*}
A verification on $\mathcal H_1= \mathcal N(A)\oplus \mathcal R(B)$ 
shows that:
\begin{enumerate}
\item  if $x\in \mathcal N(A)$, then 
$$ 
A(I+BB^*)^{-1/2} x=0 = (I+B^*B)^{-1/2} Ax;
$$

\item if $x\in \mathcal R(B)$, then there exists 
$y\in \mathcal D(B)$ such that $x=By$. 
\end{enumerate}
Thus, 
$$ 
A(I+BB^*)^{-1/2}x = A (I+BB^*)^{-1/2} B y.
$$
Moreover, since 
$$
(I+BB^*)^{-1/2}By= B (I+B^*B)^{-1/2} y,
$$
we obtain that 
$$ 
A (I+BB^*)^{-1/2}x = AB (I+B^*B)^{-1/2}y.
$$ 
On the other hand, it is not so difficult to verify that 
$$ 
\left( AB (I+B^*B)^{-1/2}y,z \right)_{\mathcal H_2} 
= \left(    (I+B^*B)^{-1/2} y, z\right)_{\mathcal H_2}
$$ 
for all $z\in \mathcal H_2$. This implies that 
$$ 
A (I+BB^*)^{-1/2}x = (I+B^*B)^{-1/2}Ax.
$$ 
Therefore, we obtain that
\begin{equation*}
\begin{split}
(I+B^*B)^{-1/2}T_{B^*} 
&=  B^*(I+BB^*)^{-1}+A(I+BB^*)^{-1} \\
&= A(I+A^*A)^{-1}+A(I+BB^*)^{-1}\\
&= A \left( (I+A^*A)^{-1}+(I+BB^*)^{-1} \right).
\end{split}
\end{equation*}
Moreover, 
$$
(I+A^*A)^{-1}x= x= (I+BB^*)^{-1} x
$$
for $x\in \mathcal N(A)$, which implies that
\begin {equation*}
\begin{split}
(I+B^*B)^{-1/2}T_{B^*}x
&=  A \left( (I+A^*A)^{-1}+(I+BB^*)^{-1} \right) x\\
&= A(2x)\\
&= 2Ax \\
&= 0.
\end{split}
\end{equation*}
For $x\in \mathcal{R}(B)$, it follows, according with 
Proposition~\ref{prop:3.1}, that
$$ 
\left( (I+A^*A)^{-1}+(I+BB^*)^{-1} \right) x = x,
$$
which implies that
$$
(I+B^*B)^{-1/2}T_{B^*}x=Ax.
$$
Hence,  
$$ 
Ax= (I+B^*B)^{-1/2})T_{B^*}x
$$
holds for all $x\in \mathcal H_1= \mathcal N(A) \oplus \mathcal R(B)$.
\end{proof}
      
% --------------------------------------

\section{\textbf{Sobolev spaces on Lipschitz domains}}
\label{sec:04} 
  
Let $\Omega$ be an open subset of $\mathbb R^d$, $d=1,2,3,\ldots$, 
$\partial \Omega$ be its boundary and $\overline{\Omega}$ its closure. 
We denote by $\mathcal C^k(\Omega)$, $k\in \mathbb{N}$ 
or $k= \infty$, the space of real $k$ times continuously 
differentiable functions on $\Omega$. The space $\mathcal{C}^{\infty}$
of all real functions on $\Omega$ with a compact support 
in $\Omega$ is denoted by $\mathcal C_c^{\infty} (\Omega)$.
For the partial differential derivatives of a function, 
we use the following notations:
$\partial_i u= \frac{\partial u}{\partial x_i}$,
$1\leq i \leq d$, for $\alpha = (\alpha_1,\ldots,\alpha_d) 
\in \mathbb N^d$, $\partial ^{\alpha} u 
= \partial_1 ^{\alpha_1}\ldots 
\partial_d ^{\alpha_d} u = \frac{\partial^{\alpha_1
+\cdots+\alpha_d}\  u}{\partial x_1^{\alpha_1}\ldots 
\partial x_d^{\alpha_d}}$ and
$|\alpha| = \alpha_1+\cdots+\alpha_d$.
For a sequence  $(\varphi_n)_{n\geq 1}$ in $\mathcal C_c^{\infty}(\Omega)$ 
and $\varphi \in \mathcal C_c^{\infty}(\Omega)$, we say that 
$(\varphi_n)_{n\geq 1} $ converges to $\varphi$  
if there exists a compact $Q\subset \Omega$ such that for all 
$n\geq 1 $ $\mathrm{supp}(\varphi_n) \subset Q$ and for all multi-index 
$\alpha \in \mathbb{N}^d$, the sequence $(\partial^{\alpha} \varphi_n)_{n\geq 1}$ 
converges uniformly to $\partial ^{\alpha} \varphi$. 
The space $\mathcal C_c^\infty(\Omega)$
induced by this convergence is denoted $\mathscr{D}(\Omega)$,
as in the theory of distributions, 
with $\mathscr{D}^{\prime}(\Omega)$ 
the space of distributions on $\Omega$.
For $k\in \mathbb N$, $H^k(\Omega)$ is the space 
of all distributions $u$ defined on $\Omega$ such that 
all partial derivatives of order at most $k$ lie in $L^2(\Omega)$, i.e., 
$$ 
\partial ^{\alpha} u \in L^2(\Omega) \ \ 
\forall  \ |\alpha | \leq k.
$$
This is a Hilbert space with the scalar product
\begin{equation*}
(u,v)_{k,\Omega} = \sum_{|\alpha | \leq k} 
\int_{\Omega} \partial ^{\alpha}u \ \overline{\partial^{\alpha}v} \ dx, 
\end{equation*}
where $dx$ is the Lebesgue measure, $u,v \in H^k(\Omega)$,
and $\overline{\partial^{\alpha}v} $ is the conjugate of $\partial^{\alpha}v$.
The corresponding norm, denoted by $\|\cdot\|_{k,\Omega}$, is given by
\begin{equation*}
\|u\|_{k,\Omega} = \left(\sum_{|\alpha| \leq k} 
\int_{\Omega} |\partial ^{\alpha}u |^2 \ dx \ \right)^{1/2}.
\end{equation*}
Sobolev spaces $H^s(\Omega)$, for non-integers $s$, 
are defined by the real interpolation method \cite{Ad,Mc,T}.

\begin{definition}
\label{def:LCB}
Let $\Omega$ be an open subset of $\mathbb R^d$ with boundary 
$\partial \Omega$ and closure $\overline{\Omega}$. 
We say that $\partial \Omega$ is Lipschitz continuous 
if for every $x\in \partial \Omega$ there exists a coordinate system 
$(\widehat{y}, y_d)\in \mathbb R^{d-1}\times\mathbb R$, 
a neighborhood $Q_{\delta,\delta'}(x)$ of $x $ and a Lipschitz function 
$\gamma_x:\widehat{Q}_{\delta} \rightarrow \mathbb R$  with 
the following properties:
\begin{enumerate}
\item $\Omega \cap Q_{\delta,\delta'}(x) 
= \left\{(\widehat{y},y_d) \in Q_{\delta,\delta'}(x) 
\ / \ \gamma_x(\widehat{x}) < y_{d} \right\}$;

\item $\partial \Omega \cap  Q_{\delta,\delta'}(x) 
=\left\{  (\widehat{y},y_d) \in Q_{\delta,\delta'}(x) 
\ / \ \gamma_x(\widehat{x}) = y_{d} \right\}$;
\end{enumerate}
where 
$$ 
Q_{\delta,\delta'}(x) 
= \left\{  (\widehat{y},y_d) \in \mathbb R^d \ / \  
\|\widehat{y}-\widehat{x}\| _{\mathbb R^{d-1}} < \delta  
\ \ \text{and} \ \ |y_d - x_d | < \delta' \right\}
$$
and 
$$ 
\widehat{Q}_{\delta}(x) = \{ \widehat{y} \in \mathbb R^{d-1} 
\ / \  \|\widehat{y}-\widehat{x}\| _{\mathbb R^{d-1}} < \delta  \} 
$$
for  $\delta, \delta' > 0$. 
An open connected subset $\Omega \subset \mathbb R^d$, whose boundary 
is Lipschitz continuous, is called a Lipschitz domain.
\end{definition}

For the rest of the paper, 
$\Omega\subset \mathbb R^d$ is a bounded Lipschitz domain.
The definitions of Sobolev spaces $H^s(\Omega)$ given above, 
remain the same for any $s$, but the spaces $H^s(\partial \Omega)$ 
can be defined by using charts on $\partial \Omega$ and partitions 
of unity subordinated to the covering of $\partial \Omega$. 
This is only possible for $|s|\leq 1$, because a Lipschitz surface 
is locally the graph of a Lipschitz function.
In particular, one frequently uses the trace spaces  
$H^1(\partial \Omega)$ and  the space of real functions 
$L^2(\partial \Omega)$
that are $L^2$ on $\partial \Omega$ 
for the surface  measure $d\sigma$.
Here we also use $H^{1/2}(\partial \Omega)$ and its dual, 
denoted by $H^{-1/2}(\partial \Omega)$.

\begin{lemma}[See, e.g., \cite{Ad}]
\label{lem:4.2}
For a bounded Lipschitz domain $\Omega$ with boundary $\partial \Omega$, 
the space $H^{1/2}(\partial \Omega)$ is dense in $L^2(\partial \Omega)$. 
\end{lemma}
 
\begin{lemma}[See, e.g., \cite{Ad}]  
\label{lem:4.3}
Let $\Omega$ be a bounded Lipschitz domain. 
Then, the space $H^s(\Omega)$ is compactly imbedded in $H^{s'}(\Omega)$ 
for all $s' < s$ in $\mathbb R$.
\end{lemma} 

% ------------------------------------------- 

\section{\textbf{Trace and embedding operators}}
\label{sec:05}

Let $\Omega \subset \mathbb{R}^{d}$, $d\geq 2$, 
be a bounded Lipschitz domain. The trace operator 
maps each continuous function $u$ on $\overline{\Omega}$ 
to its restriction onto $\partial\Omega$. 
Under the condition that $\Omega$ is a bounded Lipschitz domain, 
the trace operator may be extended to be a bounded surjective operator, 
denoted by $\Gamma_s$, from $H^s(\Omega)$ to 
$H^{s-\frac{1}{2}}(\partial\Omega)$ for $1/2<s<3/2$  
\cite{Co,Mc}. The range and null space of $\Gamma_s$ 
are given by
$$
\mathcal R(\Gamma_s)= H^{s-1/2}(\partial \Omega) 
\quad \mbox{and} \quad 
\mathcal N(\Gamma_s) = H_0^s(\Omega),
$$
respectively, where $H_0^s(\Omega)$ is defined to be the closure in 
$H^s(\Omega)$ of infinitely differentiable functions 
compactly supported in $\Omega$.
For $s=3/2$, this is no longer valid.
For $s>3/2$, the trace operator 
from $H^s(\Omega)$ to $H^1(\partial \Omega)$ is bounded \cite{Co}. 
Because of this limitation, the spaces $H^s(\Omega)$ with large $|s|$ 
are not easy to deal when considering boundary value problems 
in Lipschitz domains.	

Let us set $\Gamma=T_1 \Gamma_1$, where $\Gamma_1$ is the trace 
operator from $H^1(\Omega)$ to $H^{1/2}(\partial \Omega)$ and 
$T_1$ is the embedding operator from  $H^{1/2}(\partial \Omega)$  
into $L^2(\partial \Omega)$.  
According to a result of Gagliardo \cite{Ga}, it follows that
$\mathcal R(\Gamma) = H^{1/2}(\partial \Omega)$. Since $\Gamma_1$ 
is bounded and $T_1$ is compact \cite{N}, 
the trace operator $\Gamma$ from $H^1(\Omega)$ 
to $L^2(\partial \Omega)$ is compact. 
	
\begin{lemma}[See, e.g., \cite{Mc,N}]
\label{lem:5.1}
Let $\Gamma$ be the trace operator 
from $H^1(\Omega)$ into $L^2(\partial \Omega)$. 
Then, the adjoint operator $\Gamma^*$ is compact and injective.
\end{lemma}
	
Now, induce $H^1(\Omega)$ by the following inner product:
\begin{equation*}
(u,v)_{\partial, \Omega}= \int_{\Omega} \nabla u \nabla v dx 
+ \int_{\partial \Omega} \Gamma u \Gamma v d \sigma 
\quad \forall u,v  \in H^1(\Omega).
\end{equation*}
The associated norm $\|\cdot\|_{\partial, \Omega}$ is given by
\begin{equation*}
\|u\|_{\partial, \Omega}=\left(\|\nabla u\|^2_{0,\Omega}
+\|\Gamma u\|^2_{0,\partial \Omega} \right)^{1/2}
\end{equation*}
and 
$H^{1}(\Omega)$, induced with the inner product $(\cdot,\cdot)_{\partial,\Omega}$, 
is denoted by $H_{\partial}^{1}(\Omega)$. A further interesting remark is that
$H^{1}(\Omega)$ is the completion of $\mathcal C^1(\overline \Omega)$ with respect 
to the norm $\|\cdot\|_{1, \Omega}$. Moreover, a classical result
of Ne\v cas, asserts that under the condition $\Omega$ 
is a bounded Lipschitz domain, the norms $\|\cdot\|_{\partial, \Omega}$ 
and $\|\cdot\|_{1, \Omega}$ are equivalent \cite{N}. 
We denote by $\partial _{\nu}$ the normal derivative map, 
which maps each  $v\in C^1(\overline{\Omega})$ to 
$\partial_{\nu} v = \nu \cdot (\nabla v)_{| \partial\Omega}$ 
into $L^2(\partial\Omega)$, where $\nu$ is the unit outward normal 
on $\partial \Omega$. Under the condition $\Omega$ is a bounded Lipschitz domain, 
$\partial _{\nu}$ may be extended to a bounded linear operator, denoted by
$\widehat \partial _{\nu}$, from $H_{\Delta}^1(\Omega)$ 
to $H^{-1/2}(\partial\Omega)$, where 
$$
H_{\Delta}^1(\Omega)
= \left\{  \  v \in H^1(\Omega) \ / \ \Delta v \in L^2(\Omega) \right\}.
$$
This is a consequence of the following lemma.

\begin{lemma}[See Lemma~20.2 of \cite{T}]
\label{lem:5.2}
The application $w\longmapsto w \cdot \nu$  defined  
from $\big (\mathcal D(\overline{\Omega}) \big ) ^d $ into 
$L^{\infty} (\partial \Omega)$ is well defined and extends into 
a linear continuous map from $H_{div}(\Omega)$ into the dual space 
of $H^{1/2}(\partial\Omega)$, i.e., $H^{-1/2}(\partial\Omega)$, where 
$$
H_{div}(\Omega)= 
\left\{  w \in (L^2(\Omega))^d \ | \ \mathrm{div}\, w \in L^2(\Omega) \right\}
$$
with the norm 
$$
\|V\|_{div,\Omega}^2 = \sum_{i=1}^{d} \|v_i\|_{0,\Omega}^2 
+ \|\mathrm{div}\, v_i\|_{0,\Omega}^2 
$$
for any $V=(v_1, \ldots, v_d)\in  H_{div}(\Omega)$.
Moreover, the mapping is surjective.
\end{lemma}

As a consequence, we prove the following result.

\begin{theorem}
\label{prop:5.3}
For all $u\in H_{\Delta}^1(\Omega)$ there exists  
$\widehat \partial _{\nu} u  \in H^{-1/2}(\partial\Omega)$ 
such that 
$$ 
\int_{\Omega} \nabla u \nabla v dx = - \int_{\Omega} \Delta u \ v  \ dx 
+ \left<\widehat{\partial _{\nu}} u, \Gamma_1 v\right> 
$$
for all $v\in H^1(\Omega)$, where  $\left<\cdot,\cdot\right>$ is the duality pairing 
between $H^{\frac{1}{2}}(\partial\Omega)$ and $H^{-\frac{1}{2}}(\partial\Omega)$.
The application $u \longmapsto \widehat{\partial _{\nu}} u$ is the continuous 
extension of 
$$
u\longmapsto \nu \cdot (\nabla v)_{| \partial\Omega},
$$  
which is defined for all $u \in \mathscr{D}(\overline{\Omega})$,
where $\mathscr{D}(\overline{\Omega})$ denotes the 
space of all real valued $\mathcal C^{\infty}$
functions with a compact support on $\overline{\Omega}$.  
\end{theorem}

\begin{proof}
Let $u\in H_{\Delta}^1(\Omega)$. By setting $w=\nabla u$, we have 
$$
w\in \left( L^2(\Omega) \right) ^d \quad \mbox{and} 
\quad \mathrm{div}\, w = \mathrm{div}\, \nabla u = \Delta u \in L^2(\Omega).
$$ 
According to Lemma~\ref{lem:5.2}, there exists 
$w\cdot\nu \in H^{-1/2}(\partial\Omega)$ such that  
$$ 
\int_{\Omega} w \nabla v \ dx + \int_{\Omega} (\mathrm{div}\, w) v \ dx 
= \left<w\cdot\nu, \Gamma_1 v \right>
$$ 
for all $v \in H^1(\Omega)$ or 
$$ 
\int_{\Omega} \nabla u \nabla v dx + \int_{\Omega} \Delta u \ v  \ dx 
= \left<\widehat{\partial _{\nu}} u, \Gamma_1 v\right>,
$$
where $\left<\cdot,\cdot\right>$ is the duality pairing between 
$H^{\frac{1}{2}}(\partial\Omega)$ and $H^{-\frac{1}{2}}(\partial\Omega)$.
\end{proof}

From Theorem~\ref{prop:5.3}, we can immediately
write the following Green's formula:
	
\begin{corollary}[Green's formula]
\label{prop:GF}
Let $\Omega$ be a bounded Lipschitz domain. Then, 
$$ 
\int_{\Omega} \nabla u \nabla v dx = - \int_{\Omega} \Delta u \ Ev  \ dx 
+ \left<\widehat{\partial _{\nu}} u, \Gamma_1 v\right>
$$
for all $u \in H_{\Delta}^1(\Omega)$  and $ v \in H^1(\Omega)$,
where $E$ is the embedding operator from $H^1(\Omega)$ into $L^2(\Omega)$ and 
$\left<\cdot,\cdot\right>$ is the duality pairing between 
$H^{\frac{1}{2}}(\partial\Omega)$ and  $H^{-\frac{1}{2}}(\partial\Omega)$.
\end{corollary}

In the rest of the paper, we denote by $\widehat{g}$ the embedding 
of an element $g\in L^2(\partial \Omega)$ in $H^{-\frac{1}{2}}(\partial \Omega)$.  
Now, consider the trace operator $\Gamma$ from $H_{\partial}^1(\Omega)$ 
to $L^2(\partial\Omega)$ and let 
$\Gamma ^* \in \mathcal B(L^2(\partial \Omega), H_{\partial}^1(\Omega))$ be its adjoint.
The following result characterizes $\Gamma^*$. 

\begin{theorem}
\label{thm:5.5}
For $g \in L^2(\partial \Omega)$, $\Gamma^*$ is the solution operator 
of the following Laplace equation with Robin boundary condition:
\begin{equation*}
\begin{cases}
\Delta z =0 & \text{  }  (\Omega), \\
\partial_{\nu}z+ \Gamma z= g & \text{}  (\partial \Omega),
\end{cases}
\end{equation*}
where $\partial_{\nu}$ is the normal derivative operator, considered as non-bounded, 
from $H_{\Delta}^1(\Omega)$ to $L^2(\partial \Omega)$.
\end{theorem}

\begin{proof}
Let $g \in L^2(\partial \Omega)$ and $z=\Gamma^*g$. We have
\begin{equation}
\label{eq:star}
\begin{split}
\int_{\partial \Omega} g \ \Gamma v \  d \sigma 
&= (\Gamma^*g, v)_{ \partial ,\Omega}\\
&=  \int_{\Omega} \nabla z \nabla v dx + \int_{\partial \Omega} \Gamma z \Gamma v d \sigma, 
\end{split}
\end{equation}  
so that if $v \in  H^1_0(\Omega) ={\mathcal N}(\Gamma)$, then we obtain
$$  
\int_{ \Omega} \nabla z \nabla v \ dx = 0.
$$
Since the previous equality characterizes the $H^1$-harmonic functions, 
then 
$$ 
\Delta z = 0 \mbox{   in } \mathscr{D'}(\Omega).
$$
Applying Green's formula (Corollary~\ref{prop:GF}) 
to \eqref{eq:star}, we obtain that 
\begin{equation*}
\begin{split}
\int_{\Omega} \nabla z \nabla v dx 
+ \int_{\partial \Omega} \Gamma z \Gamma v d \sigma
&= \left<\widehat{ \partial }_{\nu}z, \Gamma_1 v\right>  
+ \int_{\partial \Omega} \Gamma z \Gamma v d \sigma \\
&=  \int_{\partial \Omega} g  \Gamma v d \sigma,
\end{split}
\end{equation*}
which leads to the following duality pairing on 
$H^{1/2}(\partial \Omega) \times H^{-1/2}(\partial \Omega)$: 
$$ 
\left<\widehat{\partial}_{\nu}z+  \widehat{\Gamma z}, \Gamma_1 v\right> 
=\left<\widehat{g},\Gamma _1 v\right>,
$$
where $\widehat{y}$ denotes the embedding of an element 
$y\in L^2(\partial \Omega)$ in $H^{-1/2}(\partial \Omega)$.
From $\mathcal R (\Gamma_1) = H^{1/2}(\partial \Omega)$, 
it follows that
$$
\widehat{ \partial} _{\nu}z + \widehat{\Gamma z} =\widehat{g}.
$$ 
Thus, $\widehat{ \partial} _{\nu}z = \widehat{g-\Gamma z}$
and, consequently, $\widehat{\partial}_{\nu} z$ belongs 
to the range of the embedding operator from $L^2(\partial \Omega)$ 
into $H^{-1/2}(\partial \Omega)$, which means that 
$\partial_{\nu}z \in L^2(\partial \Omega)$ and
$\partial_{\nu}z+ \Gamma z =g$. 
\end{proof}

Because the trace operator $\Gamma$ is bounded, one can
consider its Moore--Penrose inverse, which we denote by 
$\Lambda = \Gamma^\dagger \in {\mathcal C}(L^2(\partial\Omega), 
H_{\partial}^1(\Omega))$. 

\begin{theorem}
\label{thm:5.6}
Let $\Gamma$ be the trace operator and $\Lambda$ its Moore--Penrose inverse. 
Then, $\Lambda$ is the solution operator of the Dirichlet problem 
for the Laplace equation with data in 
$\mathcal D(\Lambda)= H^{\frac{1}{2}}(\partial \Omega)$. Moreover,
$$ 
{\mathcal D}(\Lambda)= {\mathcal R}(\Gamma),  \quad         
{\mathcal N}(\Lambda^*)= \mathcal N(\Gamma)= H_0^1(\Omega)
$$
and ${\mathcal R}(\Lambda)$ is characterized by
\begin{equation*}
\mathcal R(\Lambda)=
\left\{   v\in H^{1}(\Omega) ~~~/~~ \Delta v =0 
\ \ \mbox{in} \ \ \mathscr D '(\Omega) \right\}.
\end{equation*}
\end{theorem}

\begin{proof}  
Since $\Gamma$ is bounded, it follows that its Moore--Penrose inverse 
$\Lambda$ is closed and densely defined with closed range. 
Moreover, from Lemma~\ref{lem:5.1}, $\Gamma^*$ is injective, 
which implies that $\mathcal D(\Lambda)= \mathcal R(\Gamma)$.
Also, for $g \in \mathcal D(\Lambda)$, let $v=\Lambda g$. 
For $w \in {\mathcal D}(\Lambda^*)$, we have
\begin{equation*}
(v,w)_{\partial,\Omega}= \int_{\Omega} \nabla v \nabla w dx 
+ \int_{\partial \Omega} \Gamma v \Gamma w d\sigma 
= \left(\Lambda g,w\right)_{\partial,\Omega}
=\int_{\partial\Omega}g\Lambda^* w d\sigma,
\end{equation*}
so that if $w\in {\mathcal N}(\Lambda^*)= \mathcal N(\Gamma) =H^1_0(\Omega)$, 
then the following holds:
\begin{equation*}
\int_\Omega \nabla v\nabla w dx =0.
\end{equation*}
Since the previous equality holds for all $w \in H^1_0(\Omega)$ and  
characterizes the $H^1$-harmonic functions, it follows that
\begin{equation*}
\begin{cases}
\Delta v=0 & \text{  }  (\Omega) \\
\Gamma v =g & \text{}  (\partial \Omega).
\end{cases}
\end{equation*}
The proof is complete.
\end{proof} 

We now consider the embedding operator
$$
\begin{array}{ccccc}
E &: & H_{\partial}^1(\Omega) & \to & L^2(\Omega) \\
& & v & \mapsto & Ev, \\
\end{array}
$$
which maps each $v\in H^1(\Omega)$ to itself into $L^2(\Omega)$, 
obviously with different topologies. The space
$H^1(\Omega)$ is induced with the inner product 
$(\cdot,\cdot)_{\partial,\Omega}$ 
and $L^2(\Omega)$ with its usual inner product.
Consider also its adjoint operator $E^*$.
Since $H_0^1(\Omega) \subset \mathcal R(E)$ and $H_0^1(\Omega)$ is dense 
in $L^2(\Omega)$, it follows that $\mathcal R(E)$ is dense in $L^2(\Omega)$, 
which implies that $\mathcal N(E^*) =\{0\}$.  Also, since $\mathcal N(E)= \{0\}$, 
it follows that $\mathcal R(E^*)$ is dense in $H^1(\Omega)$.
An important characterization of $E^*$ is stated in the following result.

\begin{theorem}
\label{thm:5.7} 
Let $\Omega$ be a bounded Lipschitz domain of $\mathbb R^d$, $d\geq 2$, and $E$  
the embedding operator from $H_{\partial}^1 (\Omega)$ into $L^2(\Omega)$. Then, 
the adjoint operator $E^*$ is the solution operator of the following
Poisson equation with Robin boundary condition:
\begin{equation*}
\begin{cases}
-\Delta u =f & \text{  }  (\Omega) \\
\partial_{\nu}u+ \Gamma u=0& \text{}  (\partial \Omega),
\end{cases}
\end{equation*}
where $f\in L^2(\Omega)$.
\end{theorem}

\begin{proof}
Let $ f\in L^2(\Omega)$ and $v\in H^1(\Omega)$. Putting $u=E^*f$, one has
\begin{equation}
\label{eq:9}
\int_{\Omega} f Ev dx = (E^*f,v)_{\partial,\Omega}
= \int_{\Omega} \nabla u \nabla v \  dx 
+ \int_{\partial \Omega} \Gamma u \ \Gamma v \ d \sigma.
\end{equation}
Now, if $v \in \mathcal{ C}_ c^{\infty}(\Omega)$, then
\begin{equation*}
\begin{split}
(E^*f,v)_{\partial,\Omega} 
&=  \int_{\Omega}f Ev dx\\
&=  \int_{\Omega} \nabla v \nabla u dx\\
&= \left< -\Delta u, v \right>_{\mathscr{ D'}(\Omega),\mathscr{ D}(\Omega)}.
\end{split}
\end{equation*}
Therefore, 
$$ 
f = -\Delta u  \quad \mbox{in} \quad \mathscr D'(\Omega).
$$
Applying Green's formula (Corollary~\ref{prop:GF}) to \eqref{eq:9}, one has
\begin{equation*}
\begin{split}
\int_{\Omega} f \ Ev  dx 
&= -\int_{\Omega}  Ev \ \Delta u  \ dx 
+ \left< \widehat{\partial_{\nu}} u , \Gamma_1  v \right> 
+ \int_{\partial \Omega} \Gamma v \Gamma u \ d \sigma \\
&=   \int_{\Omega} f \ Ev  dx 
+ \left< \widehat{\partial_{\nu}} u + \widehat{\Gamma u} , \Gamma_1 v \right>.
\end{split}
\end{equation*}
So far,
$$ 
\left< \widehat{\partial_{\nu}} u + \widehat{\Gamma u} , \Gamma_1 v \right> 
=0  \quad \forall v \in H^1(\Omega).
$$
Moreover, since $\mathcal R(\Gamma_1)= H^{1/2}(\partial \Omega)$, it follows that
\begin{equation*}
\widehat{\partial_{\nu}} u + \widehat{\Gamma u}   =0 
\quad  \mbox{in} \quad  H^{-1/2} (\partial \Omega).
\end{equation*}
Consequently, $\widehat{\partial _{\nu}} u$ belongs to the range 
of the embedding operator acting from $L^2(\partial \Omega)$ to 
$H^{-1/2}(\partial \Omega)$ and $\partial _{\nu} u \in L^2(\partial \Omega)$, 
which implies that 
$$ 
\partial_{\nu} u + \Gamma u  =0 \ \mbox{ in} \   L^2(\partial \Omega). 
$$
The proof is complete.
\end{proof} 

% ------------------------------------------- 

\section{\textbf{Functional characterizations of  
$\mathbf{H^s(\partial \Omega)}$, $\mathbf{0\leq s \leq 1}$}}
\label{sec:06}

Let $\Omega$ be a bounded Lipschitz domain of $\mathbb R^d$, $d\geq 2$, 
and $\Gamma$ the trace operator from $H_{\partial}^1(\Omega)$ to 
$L^2(\partial \Omega)$. Let $\Gamma$ be bounded and 
its Moore--Penrose inverse, denoted by $\Lambda$, 
be closed and densely defined. Moreover, let the operators 
$I+\Lambda\Lambda^*$  and  $I+\Lambda ^* \Lambda$  be 
positive self-adjoint in $H^1(\Omega)$ and  $L^2(\partial \Omega)$, 
respectively. Moreover, consider they are invertible and have bounded inverses. 
Then, it makes sense to speak of their powers of any fractional order.  
Our main goal, in this section, is to make  $H^s(\partial \Omega)$ 
into a Hilbert space using the family $(I+\Lambda^* \Lambda)^{-s} $  
for real $0\leq s\leq 1$. Denote 
$$ 
\mathcal H ^{s}(\partial\Omega)
=  \left\{ \ (I+\Lambda^* \Lambda)^{-s} g 
\ | \ g\in L^2(\partial \Omega) \right\}
$$ 
for $s\geq 0$.

\begin{lemma} 
\label{lem:useDT}
Let $\Gamma$ be the trace operator from $H_{\partial}^1(\Omega)$ into $L^2(\partial \Omega)$. 
Then, the algebraic equality $\mathcal R(\Gamma)= \mathcal H^{1/2}(\partial \Omega)$ holds.
\end{lemma}

\begin{proof}
According to Theorem~\ref{OR:03}, $\Gamma$ has the following decomposition:
$$ 
\Gamma= (I+ \Lambda ^* \Lambda )^{-1/2} T_{\Lambda^*},
$$
where $T_{\Lambda^*}=\Lambda^*(I+\Lambda\Lambda^*)^{-1/2}
+\Gamma(I+\Lambda \Lambda^*)^{-\frac{1}{2}}$. Since 
$T_{\Lambda^*}$ is bounded, it follows by Douglas' theorem (see Theorem~\ref{thm:DT}) that
$$ 
\mathcal R(\Gamma) \subset \mathcal R( (I+ \Lambda ^*\Lambda )^{-1/2 })
=\mathcal H^{1/2}(\partial \Omega).
$$
On the other hand, according to Proposition~\ref{prop:3.2},
$T_{\Lambda^*}$ has a closed range and a unique 
bounded Moore--Penrose inverse $\Lambda (I+\Lambda ^* \Lambda)^{-1/2}$. 
Moreover, we have
$$  
\Gamma \Lambda (I+ \Lambda ^* \Lambda )^{-1/2} 
= (I+\Lambda ^* \Lambda)^{-1/2} T_{\Lambda^*} 
\Lambda (I+\Lambda ^* \Lambda)^{-1/2}.
$$
Now, in view of Proposition~\ref{OR:01}, we have
$$
T_{\Lambda^*} \Lambda (I+\Lambda ^* \Lambda)^{-1/2}= I_{L^2(\partial \Omega)},
$$
which implies that 
$$ 
(I+\Lambda ^* \Lambda)^{-1/2} = \Gamma \Lambda (I+\Lambda ^* \Lambda)^{-1/2}.
$$
Using again Douglas' theorem (Theorem~\ref{thm:DT}), we have 
$$
\mathcal R( (I+ \Lambda ^*\Lambda )^{-1/2}) \subset \mathcal R(\Gamma).
$$ 
Consequently, we obtain that
$$ 
\mathcal R(\Gamma) =\mathcal H^{1/2}(\partial \Omega). 
$$
The proof is complete.
\end{proof}

\begin{theorem} 
\label{thm:6.2}
The algebraic equality $H^{1/2}(\partial \Omega)
= \mathcal H^{1/2}(\partial \Omega)$ holds.
\end{theorem}

\begin{proof} 
In view of Lemma~\ref{lem:useDT}, $\mathcal R(\Gamma)= \mathcal H^{1/2}(\partial \Omega)$,
whereas according to a result of Gagliardo \cite{Ga}, 
$\mathcal R(\Gamma)= H^{\frac{1}{2}}(\partial \Omega)$.
It follows that $ H^{1/2}(\partial \Omega)
= \mathcal H^{1/2}(\partial \Omega)$.
\end{proof}

\begin{proposition}
\label{prop:usesDT}
Let $\Gamma$ be the trace operator and $\Lambda$ its Moore--Penrose inverse.
Then, the algebraic equality 
$$   
{\mathcal R}(\Gamma \Gamma^*) = \mathcal H^1(\partial \Omega)
$$
holds.
\end{proposition}

\begin{proof}
According to the third item of Proposition~\ref{prop:3.1}, we have 
$$ 
\Gamma^* (I+\Gamma \Gamma^*) ^{-1} = \Lambda (I+\Lambda^* \Lambda)^{-1}.
$$
Composing with $\Gamma$, we obtain that 
$$ 
\Gamma \Gamma^* (I+\Gamma \Gamma ^*) ^{-1} 
= \Gamma \Lambda (I+\Lambda^* \Lambda)^{-1}
$$
and, since $\Gamma \Lambda \subset I_{L^2(\partial \Omega)}$ 
and $\mathcal R((I+ \Lambda^* \Lambda )^{-1})\subset \mathcal D(\Lambda)$,  
we obtain that
$$ 
\Gamma \Gamma^* (I+ \Gamma \Gamma^*)^{-1} =(I+ \Lambda^* \Lambda )^{-1}.
$$
Therefore, by Douglas' theorem (Theorem~\ref{thm:DT}),  
$$
\mathcal R((I+ \Lambda^* \Lambda )^{-1}) \subset \mathcal R(\Gamma \Gamma^*).
$$
On the other hand, since $(I+\Gamma \Gamma^*)^{-1}$ is bounded 
and has a bounded inverse, it follows that  
$$ 
\Gamma \Gamma^* =  (I+ \Lambda^* \Lambda )^{-1} (I+ \Gamma \Gamma^*),
$$ 
which implies that 
$$
\mathcal R(\Gamma \Gamma^*)\subset \mathcal R((I+ \Lambda^* \Lambda )^{-1}).
$$ 
Hence, $\mathcal R(\Gamma\Gamma^*) = \mathcal H^1(\partial \Omega)$. 
\end{proof} 
 
The following result is a generalization of the classical theorem 
of Ne\v{c}as proved by Mclean in \cite{Mc}. 
This version will prove to be 
useful to characterize $H^1(\partial \Omega)$.
 
\begin{theorem}[See \cite{Mc}]
\label{thm:5.1}
Let $\Omega$ be a bounded Lipschitz domain of $\mathbb R^d$ 
and $u\in H _{\Delta}^1(\Omega)$.
\begin{enumerate}
\item If $\partial_{\nu} u \in L^2(\partial \Omega)$, 
then $\Gamma u \in H^1(\partial \Omega)$ and there exists 
a constant $c_{\Omega}>0$, depending on the geometry of $\Omega$, such that
$$ 
\| \Gamma u \|_ {1,\partial \Omega} \leq c _{\Omega} 
\  ( \|u\|_{1, \Omega}^2 + \|\Delta u\| _{0,\Omega}^2 
+ \| \partial  _{\nu}u \| _{0,\partial \Omega}^2)^{1/2}.
$$

\item If $\Gamma u \in H^1(\partial \Omega)$, then 
$\partial_{\nu} u \in L^2(\partial \Omega)$ and there exists 
a constant $c_{\Omega}^{\prime}>0$, depending on the geometry 
of $\Omega$, such that
$$ 
\| \partial_{\nu} u \|_ {0,\partial \Omega} \leq c_{\Omega}^{\prime} 
\  \left( \|u\|_{1, \Omega}^2 +\|\Delta u\| _{0,\Omega}^2 
+ \| \Gamma u\| _{1,\partial \Omega}^2\right)^{1/2}.
$$
\end{enumerate}
\end{theorem}
 
For $f\in L^2(\Omega)$, let us consider the solution operator $E_0^*$  
of the following Poisson equation with Dirichlet boundary condition: 
\begin{equation}
\label{eq:11}
\begin{cases}
-\Delta u^0=f & \text{  }  (\Omega) \\
\Gamma u ^0=0 & \text{}  (\partial \Omega).
\end{cases}
\end{equation}
An interesting consequence of Theorem~\ref{thm:5.1} 
is the classical Rellich--Ne\v{c}as lemma:
 
\begin{corollary}[The Rellich--Ne\v{c}as lemma -- see, e.g., \cite{N}]
Let $f\in L^2(\Omega)$ and $u^0=E_0^*f$ be the solution of the Dirichlet problem 
for the Poisson equation \eqref{eq:11}. Then, $\partial _{\nu} u^0 \in L^2(\partial \Omega)$. 
Moreover, there exists a constant $c_{\Omega}>0$, depending on the geometry of $\Omega$, such that
$$  
\|\partial_{\nu} u^0 \|_{0,\partial \Omega} \leq  c_{\Omega} \ \|f\|_{0,\Omega}.
$$
\end{corollary}

The next two theorems give new characterizations of $H^1(\partial \Omega)$.

\begin{theorem} 
\label{thm:6.6}
Let $\Gamma$ be the trace operator from $H_{\partial}^1(\Omega)$ 
into $L^2(\partial \Omega)$. Then,
$$ 
H^1(\partial \Omega) 
= {\mathcal R}((I+\Lambda^*\Lambda)^{-1}) 
= \mathcal R(\Gamma \Gamma^*).
$$
\end{theorem}
 
\begin{proof}
Let $g\in L^2(\partial \Omega)$, and $z= \Gamma^* g$ 
be the solution of the following Laplace 
equation with Robin boundary condition:
\begin{equation*}
\begin{cases}
\Delta z=0 & \text{  }  (\Omega) \\
\partial_{\nu} z + \Gamma z =g & \text{}  (\partial \Omega).
\end{cases}
\end{equation*}
Since $z\in H^1(\Omega)$ and $\Delta z= 0$ in $\Omega$,  
then $z\in H_{\Delta}^1(\Omega)$. Because $\partial _{\nu} z 
= g-\Gamma z \in L^2(\partial \Omega)$, 
by applying the first item of Theorem~\ref{thm:5.1}, 
$\Gamma z= \Gamma \Gamma^* g \in H^1(\partial \Omega)$, 
which implies that $\mathcal R(\Gamma \Gamma^*) \subset H^1(\partial \Omega)$.
Now, let $g\in H^1(\partial \Omega)$. The inclusion $H^{1}(\partial \Omega) 
\subset H^{1/2}(\partial \Omega)$ assures the existence and uniqueness 
of the variational solution of the Dirichlet problem for the Laplace equation
\begin{equation*}
\begin{cases}
\Delta v=0 & \text{  }  (\Omega) \\
\Gamma v =g & \text{}  (\partial \Omega).
\end{cases}
\end{equation*}
Since $v\in H_{\Delta}^1(\Omega)$ and $\Gamma v \in H^1(\partial \Omega)$, 
the second item of Theorem~\ref{thm:5.1} implies that $\partial_{\nu} v 
\in L^2(\partial \Omega)$. Putting $y=\partial_{\nu} v + \Gamma v$, 
it follows that $v=\Gamma ^* y$ and $g= \Gamma \Gamma^* y$. Thus, 
we deduce that $g\in \mathcal H^1(\partial \Omega)$. 
This establishes the second inclusion.
\end{proof} 

\begin{theorem}
\label{prop2:usesDT}
Let $U$ be the embedding operator from $H^1(\partial \Omega)$ 
into $L^2(\partial \Omega)$ and $V$ its inverse. Then, 
$$
H^1(\partial \Omega) = \mathcal R((I+V^*V)^{-1/2}).
$$
\end{theorem}

\begin{proof}
We have $\mathcal D(V) = \mathcal R(U)= H^1(\partial \Omega)$ and, 
by application of our Theorem~\ref{OR:03}, $U$ has the  
decomposition 
$$
U= (I+V^*V)^{-1/2} T_{V^*}.
$$
Since $U$ is injective and has a dense range, $U^*$ is injective 
and has a dense range as well, which implies that $V$ and $V^*$ 
are surjective with
$$  
\mathcal R(V^*) =\mathcal R(V^*(I+VV^*)^{-1/2})=L^2(\partial \Omega) 
$$ 
and
$$
\mathcal N(V^*) =\mathcal N (U)= \{ 0\},    
\quad \mathcal N(V)= \mathcal N(U^*)= \{ 0\},
$$
which implies that $V$ and $V^*$ are injective. It follows that 
$V^*(I+VV^*)^{-1/2}$ and $V(I+V^*V)^{-1/2}$ are isomorphisms 
from $H^1(\partial \Omega)$ into $L^2(\partial \Omega)$ and 
from $L^2(\partial \Omega)$ into $H^1(\partial \Omega)$, 
respectively (see Corollary~\ref{cor:3.6}).
Thus, $T_{V^*}$ is bounded, invertible with bounded inverse, and
$$ 
U T_{V^*}^{-1}= (I+V^*V)^{-1/2}.
$$
According to Douglas' theorem (see Theorem~\ref{thm:DT}), 
it follows that  
$$
\mathcal R(U)=\mathcal R((I+V^*V)^{-1/2}).
$$
Consequently, $\mathcal R((I+V^*V)^{-1/2})=  H^1(\partial \Omega)$.
\end{proof} 
 
Questions of equivalence of norms play an important role. 
In the rest of our paper, we adopt the notation ``$\cong$'' 
to indicate the equality between two spaces with equivalence of norms. 

\begin{theorem}
\label{cor:usesDT}
Let $\Omega$ be a bounded Lipschitz domain of $\mathbb R^d$, $d\geq 2$. Then,
$$
H^1(\partial \Omega)
\cong \mathcal H^1(\partial \Omega).
$$
\end{theorem}

\begin{proof} 
We have previously established in Proposition~\ref{prop:usesDT} 
that $\mathcal R(\Gamma \Gamma^*) = \mathcal H^1(\partial \Omega)$ 
and in Theorem~\ref{thm:6.6} that $\mathcal R(\Gamma \Gamma^*) 
=  H^1(\partial \Omega)$. Therefore, the algebraic equality
$$
\mathcal H^1(\partial \Omega) = H^1(\partial \Omega)
$$
holds. All what is needed now is to prove the equivalence of norms. 
To this end, let us consider the embedding operator $U $ from 
$H^1(\partial \Omega)$ into $L^2(\partial \Omega)$ and its inverse $V$.
According to Theorem~\ref{prop2:usesDT}, one has
$$
\mathcal D(V) = \mathcal R(U)= H^1(\partial \Omega) = \mathcal R((I+V^*V)^{-1/2}).
$$
On the other hand, we have shown in Theorem~\ref{thm:6.6} that 
$$
\mathcal R((I+\Lambda ^* \Lambda)^{-1}) = H^1(\partial \Omega).
$$ 
Therefore,
$$  
\mathcal R((I+V^*V)^{-1/2})=\mathcal R((I+\Lambda ^* \Lambda)^{-1}).
$$
By Douglas' theorem (Theorem~\ref{thm:DT}), the inclusion 
$$  
\mathcal R((I+V^*V)^{-1/2})\subset \mathcal R((I+\Lambda ^* \Lambda)^{-1})
$$ 
implies the existence of a bounded operator 
$T: L^2(\partial \Omega) \longrightarrow L^2(\partial \Omega)$ such that
$$ 
\mathcal N(T) = \mathcal N((I+V^*V)^{-1/2})= \{ 0 \} 
$$
and 
\begin{equation*}
(I+V^*V)^{-1/2 }= (I+\Lambda ^* \Lambda)^{-1} T.
\end{equation*}
Again by Douglas' theorem, Theorem~\ref{thm:DT}, the second inclusion 
$$
\mathcal R((I+\Lambda ^* \Lambda)^{-1}) \subset \mathcal R((I+V^*V)^{-1/2})
$$ 
assures the existence of a bounded operator 
$S: L^2(\partial \Omega) \longrightarrow L^2(\partial \Omega)$ such that
$$ 
\mathcal N(S) = \mathcal N((I+\Lambda ^* \Lambda)^{-1/2})= \{ 0 \} 
$$
and 
\begin{equation*} 
(I+\Lambda ^* \Lambda)^{-1} =  (I+V^*V)^{-1/2 } S.
\end{equation*}
Moreover, 
$$  
(I+\Lambda ^* \Lambda)^{-1} 
= (I+V^*V)^{-1/2 } S 
= (I+\Lambda ^* \Lambda)^{-1}  TS 
$$
and 
$$
(I+V^*V)^{-1/2 }
= (I+\Lambda ^* \Lambda)^{-1} T
= (I+V^*V)^{-1/2 } S T.
$$
This implies that $S$ and $T$ are invertible with bounded inverses 
and, more importantly, that $S$ is the inverse of $T$.
Consequently, 
\begin{equation}
\label{eq:14}
T(I+V^*V)^{1/2 }= (I+\Lambda ^* \Lambda)
\end{equation}
and 
\begin{equation}
\label{eq:15}
S (I+\Lambda ^* \Lambda) =  (I+V^*V)^{1/2 }.
\end{equation}
The equalities \eqref{eq:14} and \eqref{eq:15} imply that 
for $g\in H^1(\partial \Omega)$ the norms
$$ 
g \longmapsto \|(I+V^*V)^{1/2} g \| _{0,\partial \Omega}  
\ \ \mbox{and}   \ \  g
\longmapsto \|(I+\Lambda ^* \Lambda) g \| _{0,\partial \Omega} 
$$ 
are equivalent. Now, consider the norm $|\cdot|_{1,\partial \Omega}$ 
defined for a given $g\in H^1(\partial \Omega)$ by
$$
|g|_{1,\partial \Omega} 
= \|(I+V^*V)^{1/2 }g \|_{0,\partial \Omega}.
$$ 
Our next step is to establish the equivalence of the norms 
$|\cdot|_{1,\partial \Omega}$ and $\|\cdot\|_{1,\partial \Omega}$. 
To this end, consider $g\in H^1(\partial \Omega)$. Then, 
$Ug \in \mathcal T^1(\partial \Omega)$, where 
$$ 
\mathcal T^1(\partial \Omega) 
= \left\{ \ (I+V^*V)^{-1/2}y \ | \ y\in L^2(\partial \Omega) \right\}.
$$
It follows that,
\begin{equation*}
\|(I+V^*V)^{1/2 } Ug\|_{0,\partial \Omega } 
=  \|(I+V^*V)^{1/2 }  (I+V^*V)^{-1/2}  T_{V^*}g\|_{0,\partial \Omega }
= \|T_{V^*}g \| _{0, \partial \Omega} 
\end{equation*}
and, since $T_{V^*}$ is an isomorphism from $H^1(\partial \Omega)$ 
into $L^2(\partial \Omega)$  (see Corollary~\ref{cor:3.6}), 
there exists two constant $a,b>0$ such that  
$$
a \|g\|_{1,\partial \Omega} 
\leq |g|_{1,\partial \Omega} 
= \|T_{V^*}g \|_{0,\partial \Omega} \leq b 
\  \|g\|_{1,\partial \Omega}
$$
for all $g\in  H^1(\partial \Omega)$.
\end{proof}

\begin{corollary} 
\label{cor:6.9}
Assume $ 0\leq s\leq 1$. Then, the spaces 
$\mathcal H^s(\partial \Omega)$ form 
an interpolating family. Moreover,
$H^{s}(\partial\Omega) \cong \mathcal H ^{s}(\partial\Omega)$.
\end{corollary}

For $ 0\leq s\leq 1$, the spaces $\mathcal H^{-s}(\partial \Omega)$ 
are the dual spaces of $\mathcal H^s(\partial \Omega)$. 
This implies the following result.

\begin{corollary} 
\label{cor:6.10}
Assume $-1\leq s< 0$. Then the spaces $\mathcal H^s(\partial \Omega)$ 
form an interpolating family. Moreover,
$H^{s}(\partial\Omega)
\cong \mathcal H ^{s}(\partial\Omega)$.
\end{corollary}

% -------------------------------------------

{\bf Acknowledgments.} This research is part of first author's Ph.D. project, 
which is carried out at Moulay Ismail University, Meknes. Torres has been
supported by the R\&D Unit CIDMA and the Portuguese Foundation 
for Science and Technology (FCT), within project UID/MAT/04106/2013.

% -------------------------------------------

% -------------------------------------------

\end{document}